\documentclass[11pt,reqno]{amsart}

\usepackage{amsbsy}

\usepackage[colorlinks]{hyperref}
\hypersetup{
linkcolor=blue,          
citecolor=green,        
}
\usepackage[nobysame,abbrev,alphabetic]{amsrefs}

\usepackage{enumerate}
\usepackage{amssymb}
\usepackage{amsmath}
\usepackage{amscd}
\usepackage{amsthm}
\usepackage{amsfonts}
\usepackage{graphicx}
\usepackage[all]{xy}
\usepackage{verbatim}
\usepackage{hyperref}
\usepackage{gensymb}
\usepackage{mathrsfs}

\newtheorem{theorem}{Theorem}

\newtheorem{lemma}[theorem]{Lemma}

\theoremstyle{definition}

\theoremstyle{theorem}\newtheorem{proposition}[theorem]{Proposition}

\theoremstyle{definition}

\theoremstyle{definition}\newtheorem*{remarks}{Remarks}
\theoremstyle{definition}
\theoremstyle{definition}

\newcommand{\al}{\alpha}

\newcommand{\ga}{\gamma}
\newcommand{\Ga}{\Gamma}

\newcommand{\Del}{\Delta}

\newcommand{\Lam}{\Lambda}

\newcommand{\vphi}{\varphi}
\newcommand{\vre}{\varepsilon}
    \newcommand{\Xn}{{\mathcal{L}_n}}

\newcommand{\sm}{\smallsetminus}

\newcommand{\df}{{\, \stackrel{\mathrm{def}}{=}\, }}

\newcommand\Name[1]{\label{#1}{\ifdraft{\sn
      [#1]}\else\ignorespaces\fi}}

\newcommand\eq[2]{{\ifdraft{\ \tt
      [#1]}\else\ignorespaces\fi}\begin{equation}\label{#1}{#2}\end{equation}} 
\newcommand {\equ}[1]{\eqref{#1}}

\newcommand{\cS}{\mathcal{S}}

\newcommand{\cV}{\mathcal{V}}

\newcommand{\bR}{\mathbb{R}}
\newcommand{\bZ}{\mathbb{Z}}

\newcommand{\R}{{\mathbb{R}}}

\newcommand{\Z}{{\mathbb{Z}}}

\newcommand {\ignore}[1]  {}

\newcommand{\SL}{\operatorname{SL}}

\newcommand{\defi}{\overset{\on{def}}{=}}

\newcommand\set[1]{\left\{#1\right\}}
\newcommand\pa[1]{\left(#1\right)}

\newcommand{\E}{\mathbf{e}}

\newcommand\av[1]{\left|#1\right|}
\newcommand\on[1]{\operatorname{#1}}

\newcommand\smallmat[1]{\pa{\begin{smallmatrix}#1\end{smallmatrix}}}

\newcommand{\lra}{\longrightarrow}

\newcommand{\onto}{\xymatrix{\ar@{>>}[r]&}}
\newcommand{\da}[4]{\xymatrix{#1 \ar@<.5ex>[r]^{#2} \ar@<-.5ex>[r]_{#3} & #4}}

\newif\ifdraft\drafttrue
\draftfalse

\font\sn = cmssi8 scaled \magstep0

\begin{document}
\title{A volume estimate for the set of stable lattices}
\author{Uri Shapira}
\address{Dept. of Mathematics, Technion, Haifa, Israel
{\tt ushapira@tx.technion.ac.il}
}
\author{Barak Weiss}
\address{Dept. of Mathematics, Tel Aviv University, Tel Aviv, Israel
{\tt barakw@post.tau.ac.il}}

\maketitle
\begin{abstract}
We show that in high dimensions the set of stable lattices is almost of full measure in the space of
unimodular lattices.
\end{abstract}
\quad\\
Let $G\defi \SL_n(\bR)$, $\Ga\defi\SL_n(\bZ)$, and let $A \subset G$
denote the subgroup of diagonal matrices with positive entries. The quotient space
$\Xn\defi G/\Ga$ is naturally identified with the space of unimodular
lattices in $\bR^n$, and the group $G$ (and any of it subgroups) acts
via left translations, or equivalently, by acting on lattices via its
linear action on $\R^n$. A lattice $\Lam$ is called
{\em stable}  if for any subgroup
$\Del \subset \Lam$, one has  
$\on{vol}\pa{\Del\otimes\bR /\Del} \ge 1$ (in the literature the term
{\em semi-stable} is also used), and we denote the set of stable
lattices by $\cS^{(n)}$.

A central problem is to understand the orbits
of the $A$-action on $\Xn$.
In~\cite{SWjems} we proved that for any lattice $\Lam\in \Xn$, the
orbit-closure $\overline{A\Lam}$ contains a stable lattice. 
This result reduces the proof of Minkowski's conjecture on the product of
inhomogeneous  
linear forms to that of estimating the Euclidean covering radius of
stable lattices (see~\cite{SWjems} for details). Understanding stable
lattices is therefore a natural problem due to its connection both
with well-studied problems in the 
geometry of numbers, and with dynamics of the $A$-action. Although
$\cS^{(n)}$ is compact (while $\Xn$ is not), in this note we show that
$\cS^{(n)}$ has almost
full measure with respect to the natural probability measure on $\Xn$,
for large $n$. Moreover the convergence to full measure is very
fast. This answers a question we were asked by G. 
Harder, and can be viewed as a manifestation of the
concentration of mass along the equator in high dimensional Euclidean
balls. 

We will prove the following.
\begin{theorem}\Name{vol est theorem}
Let $m$ denote the $G$-invariant probability measure on
$\Xn$ derived from Haar measure on $G$, and let $\cS^{(n)} $
denote the subset of stable lattices in $\Xn$. Then there is a
constant $C>0$ such that for all sufficiently large $n$, 
$$
m\left(\Xn \sm \cS^{(n)} \right) \leq \left(
  \frac{C}{n}\right)^{\frac{n-1}{2}}. 
$$
In particular $m\left(\cS^{(n)} \right)\lra 1$ as $n \to \infty$. 
\end{theorem}
For $\Lam\in\Xn$ and a subgroup $\Del \subset \Lam$, we denote by
$r(\Del)$ its \textit{rank} and by $\av{\Del}$ its \textit{covolume}
in the Euclidean 
subspace $\Del\otimes \bR\subset \bR^n$.
For $k=1, \ldots,  n-1$ let us denote
$\cV_k(\Lam)\defi\set{\av{\Del}^{1/k}:\Del \subset\Lam, r(\Del)=k}$ and
$\al_k(\Lam)=\min\cV_k(\Lam)$ so that $\Lam$ is stable if and only if $\al_k(\Lam)\ge 1$ for $k=1,\dots, n-1$. 
Let 
$$
\cS^{(n)}_k(t) \defi\set{x\in \Xn: \al_k(x)\ge t}, \ \ \cS^{(n)}_k
\df \cS^{(n)}_k(1).
$$
With this notation  
$\cS^{(n)}=\bigcap_{k=1}^{n-1}\cS^{(n)}_k$. We will show: 
\begin{proposition}\Name{prop: main estimate}
There is $C>0$ such that for all sufficiently large $n$, and all $k
\in \{1, \ldots, n-1 \}$, 
\eq{1312}{
m\left(\Xn \smallsetminus \cS^{(n)}_k \right) \leq \frac1n \left(\frac{C}{n}
\right)^{\frac{k(n-k)}{2}}. 
}
\end{proposition}

\begin{proof}[Proof of Theorem \ref{vol est theorem}]
For $n>C$, the largest value of $\left( \frac{C}{n}\right)^{\frac{k(n-k)}{2}}$
is attained when $k=1$ and $k= n-1$. Therefore \equ{1312} implies
\begin{align*}
m\left(\Xn \sm \cS^{(n)}\right )&= m\left(\Xn\smallsetminus \bigcap_{k=1}^{n-1}
\cS^{(n)}_k \right) = m \left( \bigcup_{k=1}^{n-1} \Xn 
\sm \cS^{(n)}_k \right) \\
& \le 
\frac{n-2}{n} \left(\frac{C}{n} \right)^{\frac{n-1}{2}} \leq
\left(\frac{C}{n} \right) ^{\frac{n-1}{2}}. 
\end{align*}
\end{proof}

We will also show: 
\begin{proposition}\Name{prop: strengthening volume} There is $C_1>0$
  such that if we set 
\eq{eq: choice of t}{
t_k=t(n,k) \df \left(\frac{n}{C_1} \right)^{\frac{n-k}{2n} },
} 
then 
$$ \max_{k=1, \ldots, n-1} m\left(\Xn \smallsetminus
  \cS^{(n)}_k\left(t_k\right) \right) =o 
\left( \frac1n \right).
$$
In particular, $m\left(\bigcap_{k=1}^{n-1}
  \cS^{(n)}_k\left(t_k\right ) \right) \to_{n \to \infty} 1.$
\end{proposition}
\begin{remarks}
1.
Let us define
$\bar{\al}_{n,k}\defi\sup\set{\al_k(\Lam):\Lam\in\Xn}$. These
quantities are powers of 
the so-called {\em Rankin constants} or {\em generalized Hermite
constants} usually denoted by $\gamma_{n,k}$ (see \cite{Thunder}),
namely they are related by
\eq{eq: relation alpha gamma}
{\bar{\alpha}_{n,k}^{2k} = \gamma_{n,k}.
}
The origin of this exponent $2k$ is
the $1/k$ in the definition of $\mathcal{V}_k$, which we have imposed
so that the functions $\alpha_k$ behave nicely with respect to
homothety. This normalization has the additional advantage that the
growth rate of the different $\bar{\alpha}_{n,k}$ (as a function of
$n$) becomes the same for all $k$. Namely \cite[Cor. 2]{Thunder} and \equ{eq: relation
  alpha gamma} show that $\log \bar{\alpha}_{n,k} = \frac 12 \log
n + O(1)$ (where the implicit constant depends
on $k$). 

2. It seems plausible that most
lattices come close to realizing the Rankin constants, that is, for any
$\vre>0$, 
$$\lim_{n \to \infty}
m\pa{\set{\Lam\in\Xn: \forall k, \, \al_k(\Lam)>\overline{\al}_{n,k}-\vre}}=1.$$  
Combined with the result of Thunder mentioned above,  Proposition \ref{prop: strengthening volume}  may
 be viewed as supporting evidence for such a conjecture.  

3. We take this opportunity to formulate an analogous question regarding
the {\em covering radius}; that is, is it true that for any
$\vre>0$, 
$$\lim_{n \to \infty} m\set{\Lam\in\Xn:\on{covrad}(\Lam) <
  \inf_{\Lam'\in\Xn}\on{covrad}(\Lam') +\vre}=1,$$ 
where
$$\on{covrad}(\Lam)=\inf\set{r>0: \bR^n=\Lam+B(0,r)}$$ 
and $B(0,r)$ is the Euclidean ball of radius
$r$ around the origin. 
\end{remarks}

The proof of Propositions \ref{prop: main estimate} and \ref{prop: strengthening volume} relies on
Thunder's work and on a variant of Siegel's
formula~\cite{SiegelFormula} which relates the Lebesgue measure
on $\bR^n$ and the measure $m$ on $\Xn$. We now review Siegel's
method and Thunder's results.
 
In the sequel we consider $n \geq 2$ and $k \in \{1,
\ldots, n-1\}$ as fixed and omit, unless there is risk of confusion,
the symbols $n$ and $k$ from the notation.  
Consider the (set valued) map $\Phi=\Phi^{(n)}_k$ that assigns to
each lattice $\Lam\in \Xn$ the following subset of the exterior power
of $\bigwedge^k\bR^n$:
$$\Phi (\Lam)\defi\set{\pm w_\Del:\Del \subset \Lam\textrm{ a
    primitive subgroup with } r(\Del)=k},$$ 
where $w_\Del\defi v_1\wedge\dots\wedge v_k$ and $\set{v_i}_{i=1}^k$ 
is a basis for $\Del$ (note that $w_\Del$ is well-defined up to
sign, and $\Phi(\Lam)$ contains both possible choices). 
Let $$\mathscr{V} = \mathscr{V}^{(n)}_k \defi
\set{v_1\wedge\dots\wedge v_k: v_i\in\bR^n} \sm \{0\}$$ 
be the variety of pure tensors in $\bigwedge^k\bR^n$.  
For any compactly supported bounded Riemann
integrable\footnote{i.e. the measure of points at which $f$ is not
  continuous is zero.} function $f$
on $\mathscr{V}$ set  
\eq{eq: finite sum}{\hat{f}: \Xn \to \R, \ \ \ \ 
  \hat{f}(\Lam)\defi\sum_{w\in\Phi (\Lam)}f(w).}
Then it is known  (see \cite[Lemma 2.4.2]{Weil}) that the (finite) sum \equ{eq: 
  finite sum}  
defines a function in $L^1(\Xn, m)$. 
This allows us to define a Radon measure  $\theta = \theta^{(n)}_k$ on $\mathscr{V}$
by the formula 
\begin{equation}\label{1420}
\int_{\mathscr{V}} f d\theta \defi\int_{\Xn} \hat{f} \, dm, \text{  for 
  \ } f\in C_c(\mathscr{V}).
\end{equation}
Write $G=G_n \df \SL_n(\R)$. 
There is a natural transitive action of $G_n$ on $\mathscr{V}$ 
and the stabilizer of  
$e_1\wedge\dots\wedge e_k$ is the subgroup 
$$H= H^{(n)}_k \df \left\{ \smallmat{A&B\\0&D} \in G: 
A \in G_{k} , D \in G_{n-k} \right \}. $$
We therefore obtain an identification $\mathscr{V}\simeq G/H$ and view
$\theta$ as a measure on $G/H$.  
It is well-known (see e.g.~\cite{Weil}) that up to a proportionality
constant there exists a unique $G$-invariant measure 
$m_{G/H}$ on $G/H$; moreover, given Haar
measures $m_{G}, m_{H}$ on $G$ and $H$ respectively, there is a
unique normalization of $m_{G/H}$ such that 
for any $f\in L^1(G,m_G)$
\eq{1440}{
\int_G f \, dm_G =\int_{G/H}\int_{H} f(gh) dm_{H}(h)dm_{G/H}(gH).
}
We choose the Haar measure $m_G$ so that it
descends to our probability measure $m$ on $\Xn$;  similarly, we  
choose the Haar measure $m_{H}$ so that the periodic orbit
$H\bZ^n \subset \Xn$ has volume 1. These choices of Haar measures
determine our measure $m_{G/H}$ unequivocally. 
It is clear from the defining formula~\eqref{1420} that $\theta$ is
$G$-invariant and therefore 
the two measures $m_{G/H}, \theta$ are proportional. In fact (see
\cite{SiegelFormula} for the case $k=1$ and  \cite[Lemma 2.4.2]{Weil} for the general case), 
\eq{eq: Siegel normalization}{m_{G/H} = \theta.
}

For $t>0$, 
let $\chi=\chi_t:\mathscr{V}\to\bR$ be the restriction to $\mathscr{V}$ of the
characteristic function of the ball of radius $t$ around the origin,
in $\bigwedge^k\bR^n$, with respect to the natural inner product
obtained from the Euclidean inner product on $\R^n$. 
Note that  
$\hat{\chi}(x)=0$ if and only if $x\in \cS^{(n)}_k\left(t^{1/k}\right)$ and furthermore,
$\hat{\chi}(x)\ge 1$ if $x\in \Xn\smallsetminus
\cS^{(n)}_k\left(t^{1/k} \right)$.  
It follows that
\eq{eq: using chi}{m\left(\Xn\smallsetminus
\cS_k^{(n)}(t)\right)\le\int_{\Xn}\widehat{(\chi_{t^k})} dm =\int_{\mathscr{V}}\chi_{t^k}
d\theta.
}

Let $V_j$ denote the volume of the Euclidean unit ball in
$\bR^j$ and let $\zeta$ denote the Riemann zeta function.  We will use
an unconventional convention $\zeta(1)=1$, which will make our
formulae simpler. 
For $j \geq 1$, define 
$$
R(j) \df 
 \frac{j^2 V_j}{ \zeta(j)} 
\quad\textrm{and}\;\; 
B( n,k)\defi \frac{\prod_{j=1}^nR(j)}{\prod_{j=1}^k R(j)\prod_{j=1}^{n-k}R(j)}.$$
The following is \cite[Lemma 5]{Thunder}:
\begin{theorem}[Thunder]{\label{Thunder}}
For $t>0$, we have 
$\int_{\mathscr{V} } \chi_t \, dm_{G/H}
=B( n,k)\frac{ t^n}{n}.$ 
\end{theorem} 
\noindent (Note that in Thunder's notation, by \cite[\S4]{Thunder},
$c(n,k)=B(n,k)/n$.)

\medskip

We will need to bound $B( n,k)$. 
\begin{lemma}\Name{lem: bound on Bin}
There is $C> 0$ so that for all large enough $n$ and
all $k=1, \ldots, n-1$, 
\eq{eq: bound on Bin}{B( n,k)\leq
  \left(\frac{C}{n}\right)^{\frac{k(n-k)}{2}}.
}
\end{lemma}
\begin{proof}In this proof $c_0, c_1, \ldots$ are constants independent of $n, k, j$. 
Because of the symmetry $B(n,k)=B(n,n-k)$ it is
enough to prove \equ{eq: bound on Bin} with $k\leq \frac{n}{2}.$
Using 
the formula
$V_j=\frac{\pi^{j/2}}{\Ga\left(\frac{j}{2}+1\right)}$ we obtain 
%
\begin{align*}
B(n,k)&=\prod_{j=1}^k\frac{R(n-k+j)}{R(j)}
=\prod_{j=1}^k\frac{\zeta(j)(n-k+j)^2\frac{\pi^{(n-k+j)/2}}{\Ga(\frac{n-k+j}{2}+1)}}
{\zeta(n-k+j)j^2\frac{\pi^{j/2}}{\Ga(\frac{j}{2}+1)}}\\
&=\prod_{j=1}^k \frac{\zeta(j)}{\zeta(n-k+j)}\cdot\pa{\frac{n-k+j}{j}}^2\cdot\pi^{\frac{n-k}{2}}\cdot
\frac{\Ga(\frac{j}{2}+1)}{\Ga(\frac{n-k+j}{2}+1)}.
\end{align*}
Note that $\zeta(s) \geq 1$ is a decreasing function of $s>1$, so
(recalling our convention $\zeta(1)=1$) 
$\frac{\zeta(j)}{\zeta(n-k+j)} \leq c_0 \df \zeta(2)$.
It follows that for all large enough $n$ and 
for any $1\le j\le k, $ 
\eq{eq: estimate first part}{
\frac{\zeta(j)}{\zeta(n-k+j)}\cdot\pa{\frac{n-k+j}{j}}^2\cdot\pi^{\frac{n-k}{2}}\le c_0
n^2\pi^{\frac{n-k}{2}}\le 4^{\frac{n-k}{2}}.
}
According to Stirling's formula, there are positive constants $c_1,
c_2$ such that for all $x \geq 2$, 
$$
c_1 \sqrt{\frac{2\pi}{x}}\left(\frac{x}{e} \right)^x \leq \Gamma(x)
\leq c_2 \sqrt{\frac{2\pi}{x}}\left(\frac{x}{e} \right)^x.
$$
We set $u \df \frac{j}{2}+1$ and $v \df \frac{n-k}{2} $, so that 
$u+v \geq \frac{n-1}{4}$, 
and obtain  
\eq{eq: estimate second part}{
\begin{split}
\frac{\Ga(\frac{j}{2}+1)}{\Ga(\frac{n-k+j}{2}+1)} & =
\frac{\Ga(u)}{\Ga(u+v)} \leq \frac{c_2}{c_1}
\sqrt{\frac{u+v}{u}}\frac{u^u}{(u+v)^{u+v}} \frac{e^{u+v}}{e^u} \\
& \leq c_3 e^v \frac{u^{u-1/2}}{(u+v)^{u+v-1/2}} = c_3
\left(\frac{e}{u+v}\right)^v \frac{1}{\left(1+\frac{v}{u}
  \right)^{u-1/2}}, 
\\
& \leq c_3 \left( \frac{4e}{n-1} \right)^{\frac{n-k}{2}}. 
\end{split}
}
Using \equ{eq: estimate first part} and \equ{eq: estimate
  second part} we obtain
$$
B( n,k) \leq \left[c_3 4^{\frac{n-k}{2}}
  \left(\frac{4e}{n-1}\right)^{\frac{n-k}{2}} \right]^k = \left[ c_3
  \left(\frac{16e}{n-1} \right)^{\frac{n-k}{2}} \right]^k.
$$
So taking $C > 16c_3 e$ 
we obtain \equ{eq: bound on
  Bin} for all large enough $n$. 
\end{proof}
\begin{proof}[Proof of Propositions \ref{prop: main
    estimate} and \ref{prop: strengthening volume}]
Let $C$ be as in Lemma \ref{lem: bound on Bin} and let $C_1>C$. For
Proposition \ref{prop: strengthening volume}, using 
\equ{eq: using chi}, \equ{eq: Siegel normalization} and 
Theorem~\ref{Thunder}, for all sufficiently large $n$ we have
\[\begin{split}
m\left(\Xn \smallsetminus \cS^{(n)}_k(t_k) \right) & \leq
B(n,k) \frac{t_k^{kn}}{n} \\
& \leq \frac1n \left(\frac{C}{n} \right)^{\frac{k(n-k)}{2}}
\left(\frac{n}{C_1} \right)^{\frac{k(n-k)}{2}} =  \frac1n \left(\frac{C}{C_1} \right)^{\frac{k(n-k)}{2}}.
\end{split}
\]
Multiplying by $n$ and taking the maximum over $k$ we obtain 
$$
n \, \max_{k=1, \ldots, n}  m\left(\Xn \smallsetminus
  \cS^{(n)}_k(t_k) \right) \leq \left(\frac{C}{C_1}
\right)^{\frac{n-1}{2}} \to_{n\to \infty} 0.
$$
The proof of Proposition \ref{prop: main estimate} is identical using
$t=1$ instead of $t_k$. 
\end{proof}

\ignore{
\begin{proof}[Proof of Proposition \ref{prop: main estimate}]
Let $C$ be as in Lemma \ref{lem: bound on Bin}. 
As above, 
for all sufficiently large $n$ we have
\[
m\left(\Xn \smallsetminus \cS^{(n)}_k \right) 
\leq
\frac{B(n,k) }{n} 
\leq \frac1n \left(\frac{C}{n} \right)^{\frac{k(n-k)}{2}} \leq
\frac{1}{n} \left( \frac{C}{n}\right)^{\frac{n-1}{2}}. 
\]
Hence by \equ{}, 
$$
1-m\left(\cS^{(n)} \right) \leq \sum_{k=1}^{n-1} m \left(\Xn \sm
  \cS^{(n)}_k \right) \leq \left(\frac{C}{n} \right )^{\frac{n-1}{2}}.
$$
\end{proof}
}

{\bf Acknowledgements.} The authors thank Prof. G\"unter Harder for
an 
interesting conversation regarding stable lattices and for 
raising the question addressed in this article. The work of the
authors was supported by ERC starter grant DLGAPS 279893 and ISF
grants 190/08, 357/13, and the Chaya Fellowship. The results 
of this paper appeared on arXiv as part of a preliminary version of
\cite{SWjems}.

\def\cprime{$'$} \def\cprime{$'$} \def\cprime{$'$}

\end{document}